\documentclass[a4paper,UKenglish,cleveref, autoref, thm-restate]{lipics-v2021}
  
\usepackage[inline]{enumitem}
  \usepackage{cite}
  \usepackage[utf8]{inputenc}
\usepackage[T1]{fontenc}
\usepackage{amsmath,amsthm,amssymb,amsfonts}
\usepackage{mathtools}
\usepackage{xspace}
\usepackage{hyperref}
\usepackage{cleveref}
\usepackage{verbatim}
\usepackage{tikz}
\usepackage{setspace}
\usetikzlibrary{snakes}
\usepackage[procnumbered,linesnumbered,ruled,vlined]{algorithm2e}

	\newenvironment{subproof}[1][\proofname]{%
	\begin{proof}[#1]%
	}{%
	\end{proof}%
}

\hypersetup{
    colorlinks=true,
    citecolor=red,
    linkcolor=blue,
    filecolor=magenta,      
    urlcolor=black,
}
\usetikzlibrary{positioning, fadings, backgrounds}

\usepackage[textsize=footnotesize,color=green!40]{todonotes}

\makeatother

\newcommand{\heightn}[1]{\ell_{#1}}
\usetikzlibrary{decorations.pathmorphing}
\tikzset{snake it/.style={decorate, decoration=snake}}

\newcommand{\dist}[2]{d(#1,#2)}
\newcommand{\distance}[3]{d_{#1}\left(#2,#3\right) }

\newcommand{\trc}[2]{trace_{#1}\left(#2\right)}

\newcommand{\ipco}[1]{ipco\left(#1\right)}
\newcommand{\sipco}[1]{sipco\left(#1\right)}
\newcommand{\diam}{\mathrm{diam}}

\newcommand{\downboundary}[1]{\delta^{\downarrow}_{#1}}
\newcommand{\upboundary}[1]{\delta^{\uparrow}_{#1}}

\makeatletter
\newcounter{claimcnt}[theorem]

\theoremstyle{plain}
\newtheorem{claimx}[claimcnt]{Claim}
\AtBeginDocument{%

}
\makeatother

\nolinenumbers

\newtheorem{question}[theorem]{Question}

\definecolor{dartmouthgreen}{rgb}{0.05, 0.5, 0.06}

\title{Additive Quasi-isometry via rooted graph partitions and layering partition}

\author{Dibyayan Chakraborty}{School of Computer Science, University of Leeds, United Kingdom.}{}{}{}

\author{Yann Vax\`{e}s}{Laboratoire d'Informatique et Systèmes, Aix-Marseille Université and CNRS, Faculté des Sciences de Luminy, F-13288 Marseille, Cedex 9, France.}{}{}{}



\authorrunning{D. Chakraborty, Y. Vax\`{e}s}

\titlerunning{Additive Quasi-isometry via rooted graph partitions}

\begin{document}

\maketitle

\setstretch{1}

\begin{abstract}
For a graph $H$, $\langle H \rangle$ denotes the class of all subdivisions of 
$H$ and $tw(H)$ denotes the treewidth of $H$. In this paper, we prove the following. For $k\geq 1, R\geq 1$, let $G,H$ be two graphs such that strong isometric path complexities (Chakraborty et al. [\textsc{Disc. Math., 2026}]) of both $G$ and $\langle H \rangle$ are at most $k$, and $G$ admits an honest, ``nicely rooted'' $R$-bounded $H$-partition.   Then, there is a graph $F$ with $tw(F)\leq tw(H)$ such that $G$ admits a $(1,33\cdot R\cdot k^2)$-quasi-isometry to $F$.  	
	Using results of Albrechtsen, Distel, and Georgakopoulos (2025), we also obtain that $K_{2,t}$-asymptotic minor-free graphs admit quasi-isometries with additive distortion to $K_{2,t}$-minor-free graphs. This answers an open question raised by the above authors. 
	
	As part of our proof, we combine the graph-partition based method and the layering partition based method (Chepoi et al. [\textsc{Discrete Comput. Geom.} 2012]) to obtain additive quasi-isometry when both the source and all subdivisions of the target graph have bounded strong isometric path complexity. 
\end{abstract}

\section{Introduction}
In this paper, we consider graphs that are finite and, unless otherwise stated, connected and unweighted.
 
When is a graph \emph{quasi-isometric} (\Cref{def:quasi}) to a much simpler \emph{host} graph? This question is central to \emph{coarse graph theory}~\cite{georgakopoulos2025graph}, an area that has seen much development recently~\cite{georgakopoulos2025graph,albrechtsen2023structural,albrechtsen2025characterisation,albrechtsen2025excluding,albrechtsen2025counterexample,albrechtsen2026small,davies2026fat}. Inspired by Bonamy et al.~\cite{bonamy2023asymptotic} and Chepoi et al.~\cite{chepoi2012constant}, Georgakopoulos and Papasoglu~\cite{georgakopoulos2025graph} introduced the notion of \emph{asymptotic minors} (see Definition~\ref{def:asym}), a coarse generalisation of graph minors.  The following conjecture was posed in the same paper. 

\begin{conjecture}
Let $X$ be a graph, and let $H$ be a finite graph. Then $X$ has no $K$-fat $H$ minor for some $K \in \mathbb{N}$ if and only if $X$ is $f(K)$-quasi-isometric to a graph with no $H$-minor, where $f : \mathbb{N} \rightarrow \mathbb{N}$ depends on $H$ only.	
\end{conjecture}

While the above conjecture has been proved to be false~\cite{davies2026fat,albrechtsen2025counterexample,albrechtsen2026small} in general, it has been shown to hold for small graphs, e.g., $H=K_3$\cite{berger2024bounded,dragan2025graph,dourisboure2007tree}, $H=K_{2,3}$~\cite{chepoi2012constant} (or equivalently $H=K_4^-$~\cite{albrechtsen2025characterisation}), $H=K_4$~\cite{albrechtsen2025characterisation}. More recently, Albrechtsen, Distel, and Georgakopoulos~\cite{albrechtsen2025excluding} settled the above conjecture for $H=K_{2,t}$ for every $t\geq 3$.  

\sloppy When quasi-isometric maps with only \emph{additive distortions} (\Cref{def:add}) are allowed, the situation is much less understood. Brandst{\"a}dt et al.~\cite{brandstadt1999distance} proved that \emph{chordal} graphs (i.e., graphs without any induced cycle of length greater than three) admit a quasi-isometry to trees with additive distortion at most $2$.  Berger \& Seymour~\cite{berger2024bounded} characterized graphs admitting quasi-isometry with additive distortion to trees. See~\cite{dragan2025graph} for a comprehensive overview. This line of research is also motivated by the following question.

	\begin{question}[\cite{georgakopoulos2025graph}]\label{quest:add}
	For which finite graphs $H$ it is true that if a graph $X$ has no $K$-fat $H$-minor for some $K \in \mathbb{N}$, then $X$ admits a map of bounded additive distortion onto a graph with no $H$ minor?
\end{question}

Albrechtsen, Distel, and Georgakopoulos~\cite{albrechtsen2025excluding} specifically asked the above question for $H=K_{2,t}$. 

	\begin{question}[\cite{albrechtsen2025excluding}]\label{quest:K_2t}
	Is it true that if a graph $X$ has no $K$-fat $K_{2,t}$-minor for some $K \in \mathbb{N}, t\geq 3$, then $X$ admits a map of bounded additive distortion onto a graph with no $K_{2,t}$-minor?
\end{question}

In this paper, we answer the  Question~\ref{quest:K_2t} in the affirmative. Before we state our main results, we revisit some of the techniques often used to obtain quasi-isometry. 

%
%
%
%

\paragraph*{Graph partitions}
 
 Perhaps the most widely used method for obtaining quasi-isometry is based on \emph{graph partitions}. Informally, a graph partition of a graph $G$ over a graph $H$, also called $H$-partition, is a partition $\mathcal{P}$ of $V(G)$ along with a mapping $\mathcal{P}\rightarrow V(H)$ such that the endpoints of any edge of $G$ either lie inside the same part or lie in two parts whose corresponding vertices in $H$ are adjacent. Furthermore, the graph partition over $H$ is \emph{honest} if any edge of $H$ is also realised by an edge of $G$. The graph partition is $R$-bounded for some integer $R$ if the maximum distance between two vertices in the same partition in $G$ is at most $R$. The following lemma has been central to obtaining quasi-isometry.

 \begin{lemma}[\cite{albrechtsen2025excluding,albrechtsen2023structural}]
 	Let $H$, $G$ be graphs, and let $\mathcal{H}$ be an honest, $R$-bounded $H$-partition of $G$ for some $R \in \mathbb{R}$. Then $G$ is $(R + 1, R/(R + 1))$-quasi-isometric	to $H$.
 \end{lemma}
 
 Due to the above lemma, researchers focussed mainly on obtaining the right graph partitions. Most of the proofs (based on graph partitions) have the following approach. They start with an arbitrarily chosen vertex $r$ called the ``root'' and a radius $R$ chosen appropriately based on the forbidden graph $H$ and the fatness parameter \(K\). Then, they consider all components of $G$ obtained by removing the ball with radius $R$ centered on $r$. This ball itself forms one of the partitions. Then based on how these components attach to the ball, specific strategies are devised on how to obtain the next partitions. For example, Albrechtsen et al.~\cite{albrechtsen2025characterisation} use the Coarse Menger's Theorem, whereas the strategy of Albrechtsen, Distel, and Georgakopoulos~\cite{albrechtsen2025excluding} exploits the structure of $K_{2,t}$-minor-free graphs. We view this approach as a ``bottom up''\footnote{We follow the view in which the root is considered to be at the bottom of a tree} approach where the process starts at a ball centered around a vertex and then proceeds by considering components remaining after deleting the vertices of the ball. At each step a new part is constructed and it is ensured that it contains at least one endpoint of an edge in $G$ whose other endpoint is contained in a previously constructed part. At the end of the procedure (for a finite graph $G$), this results in an honest partition.

 \paragraph*{Layering partition}
 
 A relatively less explored method is that of \emph{layering partition}~\cite{brandstadt1999distance,chepoi2000note}.  See \Cref{sec:prelim} for definitions. Informally, in 
 a layering partition, the vertex set of a graph is partitioned into \emph{clusters} w.r.t a fixed \emph{root} vertex. Two vertices are in the same cluster if they are at a distance $k$ from the root for some $k\geq 0$ and are end-vertices of a path that is disjoint from the ball of radius $k-1$ centered at 
 the root vertex. It was shown in~\cite{brandstadt1999distance,chepoi2000note} that the adjacency relation between the clusters is captured by a tree, known as the \emph{layering tree}. Moreover, for a fixed root vertex, the corresponding layering tree can be constructed in linear time~\cite{chepoi2000note}.
 
 Chepoi et al.~\cite{chepoi2012constant} proved that $K_{2,3}$-asymptotic minor-free graphs admit quasi-isometries to cactus graphs using layering partitions. Essentially, they processed the vertices  ``top-down'', starting from the leaves of the layering tree. They characterised the distance properties of vertices lying inside a cluster (when the source graph has no $K$-fat $K_{2,3}$-minor) and proved that each of these clusters can be partitioned into at most two balls that are ``far apart''. Then, they used these properties to build the target cactus exactly on the vertices of $G$. 
 
 An interesting aspect of the above approach is that it preserves the rooted distance of the vertices in the cactus. This raises the question: If every isometric path of the host can be covered by a bounded number of isometric paths emanating from the root, can the method be adapted to obtain a quasi-isometry with additive distortion?
 
 \paragraph*{Strong isometric path complexity}
 
The \emph{strong isometric path complexity} is a recently introduced graph invariant that captures how arbitrary isometric paths of a 
graph can be covered by a bounded number of ``rooted'' isometric paths (i.e. {isometric} paths with a common end-vertex). See \Cref{def:strong}. Informally, if the strong isometric path complexity of a graph $G$ is at most $k$, then for any isometric path $P$ of $G$ and any vertex $v\in V(G)$, the vertices of $P$ can be covered with $k$-many $v$-rooted isometric paths of $G$. 	Originally, strong isometric path complexity and its weaker variant isometric path complexity were introduced with an algorithmic motivation~\cite{c22isom,chakraborty2026isometric}. Later, it was shown in~\cite{chakraborty2025strong} that $K_{2,t}$-asymptotic minor-free graphs have bounded strong isometric path complexity. On the other hand, graphs with tree-width at most two have unbounded strong isometric path complexity. The notions of layering partitions and strong isometric path complexity have been used to obtain additive quasi-isometry for $K_{2,3}$-induced minor-free graphs to treewidth $2$ graphs~\cite{chakraborty2025k}. Recently, Papasoglu and Swenson~\cite{papasoglu2026additivequasiisometriescacti} have proved a more general result for $K_{2,3}$-asymptotic minor-free graphs.  
Interestingly, the authors' use of the notion ``monotone-height path'' is reminiscent of rooted paths and strong isometric path complexity.

\paragraph*{Our results}

In this paper, we show that the graph partition based ``bottom-up'' approach can be combined with the layering partition based ``top-down'' approach to obtain quasi-isometry when the source and all subdivisions of the target graph have bounded strong isometric path complexity. We formally state our main result. For a graph $H$, let $\langle H \rangle$ be the class of graphs that can be obtained by subdividing the edges of $H$. 

\begin{theorem}\label{thm:additive}
	For $k\geq 1, R\geq 1$, let $G,H$ be two graphs such that $\sipco{G}\leq k, \sipco{\langle H \rangle}\leq k$, and $G$ admits an honest, nicely rooted $R$-bounded $H$-partition.  There is a graph $F$ with $tw(F) \leq tw(H)$  and $G$ admits a $(1,33\cdot R \cdot k^2)$-quasi-isometry to $F$. 
\end{theorem}

The above theorem is similar in spirit to that of Lemma~\ref{lem part QI}, but we only manage to preserve the treewidth of the quotient graph. We note that the definition of ``nicely rooted graph partition'' (see Definition~\ref{def:nice}) is motivated by a similar notion introduced by Albrechtsen, Distel, and Georgakopoulos~\cite{albrechtsen2025excluding}. However, our definition  assumes certain conditions on the height of the nodes of $H$ that lie in the same cluster in the layering partition of $H$, which is a deviation.

\paragraph*{Overview of the proof.} Let $G$ be a graph that admits an honest, nicely rooted $R$-bounded $H$-partition, for some graph $H$. Further assume that $G$ and all subdivision of bounded strong isometric path complexity. To construct the graph $F$ we choose a "representative" vertex from each part of the graph partition, and for every edge of $H$ we introduce a path between the corresponding representatives. For other vertices inside a part, we arbitrarily select one "parent" of the corresponding representative and introduce a path between the concerned vertex and the path. Lengths of all paths equals the difference of ``levels'' of the endpoints in $G$. To prove that the graph constructed above indeed serves its purpose, we make the following observations. 

We show (in Lemma~\ref{lem:sipco-distort}) for some graphs $X$ and $Y$ where $X$ has bounded strong isometric path complexity, if there is a map from $V(X)$ to $V(Y)$ such that the endpoints of edges of $X$ remain close-by and the rooted paths are preserved upto an additive distortion, then bounded strong isometric path complexity of $X$ implies a similar property for all isometric paths of $X$. 

Lemma~\ref{lem:sipco-distort} reduces the problem to only analysing the distortion between ``rooted pairs'' in $G$ and $F$. Indeed, it is not difficult to prove using the properties of nicely, rooted graph partitions that the distances between any two ``rooted pairs'' of vertices of $G$ increase by a constant additive error in $F$. This constant is even independent of the strong isometric path complexity of $G$, but depends on the diameter of the partitions of $H$. 

However, it is slightly more tricky to prove the converse.
First, we show (in Lemma~\ref{lem:attach-ipco}) that since all graphs in $\langle H\rangle$ have bounded strong isometric path complexity, for any path $P$ of $F$, any component of $F-V(P)$ sees into at most $k$ vertices of $P$. 
Using the above observation, we show , that there is a mapping from $F$ to $G$ where distances between rooted pairs in $F$ increase (in $G$) by a linear function of the the diameter of the partitions of $H$ and strong isometric path complexity of $\langle H\rangle$. {
 Now we conclude the proof  again by using  Lemma~\ref{lem:sipco-distort}.
}
\paragraph*{A pleasant consequence} 
We observe that for any graph $G$ with no $K$-fat $K_{2,t}$-minor, the proof of Albrechtsen, Distel, and Georgakopoulos~\cite{albrechtsen2025excluding} already provides a  nicely rooted $R$-bounded $H$-partition for some $R$ where $H$ is $K_{2,t}$-minor-free. Now combining with results from Chakraborty and Foucaud~\cite{chakraborty2025strong}, we have the following theorem.

\begin{theorem}\label{thm:K2t}
There is a function $g:\mathbb{N}\times \mathbb{N}\rightarrow \mathbb{N}$ such that for any graph $G$ without any $K$-fat $K_{2,t}$-minor for $t\geq 3$, there is a $K_{2,t}$-minor-free graph to which $G$ admits a $(1,g(K,t))$-quasi-isometry. 
\end{theorem}

\section{Preliminaries}\label{sec:prelim}

Let $G$ be a graph. For a set $X\subseteq V(G)$, $G-X$ is the graph induced by $V(G)\setminus X$.
 The length of a path $P$ is the number of edges in $P$. 
 For two vertices $u,v\in V(G)$, a $(u,v)$-path (resp. $(u,v)$-induced path) is a path (resp. an induced path) between $u$ and $v$. 
 A $(u,v)$-isometric path is a $(u,v)$-path with smallest possible length. 
 The \emph{distance} between two vertices $u,v\in V(G)$, denoted by $\distance{G}{u}{v}$, is the length of a $(u,v)$-isometric path in $G$. 
 For a set $S$ and a vertex $u$, the distance between $u$ and $S$ is $\distance{G}{u}{S}=\min\{\distance{G}{u}{v}\colon v\in S\}$.	
For a vertex $r\in V(G)$, an ordered pair $(u,v)$ with $u,v\in V(G)$, is an \emph{$r$-rooted pair} (resp. \emph{strongly $r$-rooted}) in $G$ if $u$ lies in some (resp. all) $(r,v)$-isometric path(s) of $G$. 

Given a set~$U$ of vertices of~$G$, the \emph{ball (in~$G$) around~$U$ of radius $r \in \mathbb{N}$}, denoted by~\emph{$B_G(U, r)$}, is the set of all vertices in~$G$ of distance at most~$r$ from~$U$ in~$G$.
If~$U = \{v\}$ for some~$v \in V(G)$, then the braces are omitted, writing~$B_G(v, r)$ instead of $B_G(\{v\}, r)$. 

The \emph{diameter} of $G$, denoted as \emph{$\diam(G)$}, is the smallest number~$k \in \mathbb{N} \cup \{\infty\}$ such that $d_G(u,v) \leq k$ for every two~$u,v \in V(G)$. If $G$ is empty, then we define its diameter to be~$0$.
We remark that if $G$ is disconnected but not the empty graph, then its diameter is $\infty$.
The \emph{diameter of a set $U\subseteq V(G)$ in $G$}, denoted by \emph{$\diam_G(U)$}, is the smallest number $k \in \mathbb{N}$ such that $d_G(u,v) \leq k$ for all $u,v \in U$ or $\infty$ if such a $k \in \mathbb{N}$ does not exist. If $Y$ is a subgraph of $G$, then we abbreviate $d_G(U,V(Y))$, $\diam_G(V(Y))$ and $B_G(V(Y),r)$ as $d_G(U,Y)$, $\diam_G(Y)$ and $B_G(Y,r)$, respectively.

	We recall the notion of \emph{layering partition}, introduced in \cite{brandstadt1999distance, chepoi2000note}. Let $G$ be a graph and $r$ be any vertex of $G$. A subset $S \subseteq V(G)\setminus \{r\}$ is a \emph{cluster w.r.t $r$} if $S$ is the maximal subset such that for any two vertices $u,v\in S$, (i) $d_G(r,u)=d_G(r,v)$ and (ii) $u,v$ lie in the same connected component in $G-B_G\left( r, k-1\right)$ where  $k=d_G(r,S)$. 
	
	Let $G$ be any graph. A chordless path $P$ of $G$ between two vertices $u$ and $v$ is a \emph{pendant} path if $u$ has degree $1$ in $G$, and every vertex in $V(P)\setminus \{u,v\}$ has degree $2$ in $G$. 
	
	For a graph $G$ and two vertices $u,v\in V(G)$, a set $X\subseteq V(G)$ is a \emph{$(u,v)$-separator} if $u$ and $v$ lie in different components of $G-X$.

A \emph{tree decomposition} of a graph $G$ is a pair $(T, (X_t)_{t \in
	V(T)})$, 
where $T$ is a rooted tree and $(X_t)_{t \in
	V(T)}$ is a family of subsets of $V(G)$ associating one subset
$X_t$, called \emph{bag}, to every $t \in V(T)$ such that:
\begin{itemize}
	\item[(i)] the set of nodes of $T$ containing a given vertex of $G$
	forms a nonempty connected subtree of $T$, and 
	\item[(ii)] any two
	adjacent vertices of $G$ appear together in a common node of $T$.
\end{itemize}

The \emph{width} of $T$ is the maximum cardinality of a bag minus one. The \emph{treewidth} of $G$, denoted as $tw(G)$, is the minimum integer $k$ such that $G$ has a tree decomposition of width $k$. 
\subsection{Asymptotic minors, quasi-isometry, and graph partitions}

For a positive integer $K$ and a graph $H$, a \emph{$K$-fat minor model} of $H$ in a graph $G$ is a collection $\mathcal{M}=\left(B_v\colon v\in V(H)\right) \cup \left(P_e\colon e\in E(H)\right)$ of connected subgraphs of $G$ such that \begin{itemize}
	\item $V(B_v) \cap V(P_e)\neq \emptyset$ whenever $v$ is an end-vertex of $e$ in $H$;
	\item for any pair of distinct $X,Y\in \mathcal{M}$ not covered by the above condition, we have $\distance{G}{X}{Y} \geq K$. 
\end{itemize}  

The subgraphs in $(B_v\colon v\in V(H))$ are the ``branch sets'' of the model. The subgraphs in $(P_e\colon e\in E(H))$ are the ``branch paths'' of the model. If $G$ contains a $K$-fat minor model of a graph $H$, then we say that $H$ is a $K$-fat minor of $G$. Note that, for $e\in E(H)$, we can assume that the connected subgraphs $P_e$ are indeed paths.

\begin{definition}[\cite{georgakopoulos2025graph}]\label{def:asym}
	A graph class $\cal G$ contains $H$ as an \emph{asymptotic minor} if for every integer $K\geq 1$, there exists a graph $G\in \cal G$, such that $G$ contains $H$ as a $K$-fat minor.
\end{definition}

We say that {a graph class} $\mathcal{G}$ is \emph{$H$-asymptotic minor-free} if $\cal G$ does not contain $H$ as an asymptotic minor. In other words, $\cal G$ is $H$-asymptotic minor-free if there exists a constant $K$, depending only on $\mathcal{G}$, such that no graph in $\mathcal{G}$ contains a $K$-fat minor model of $H$.

\begin{definition}[\cite{georgakopoulos2025graph}] \label{def:quasi}
	\sloppy An \emph{$(M,A)$-quasi-isometry} between graphs $G$ and $H$ is a map $f\colon V(G)\rightarrow V(H)$ such that the following holds for fixed constants $M\geq 1, A\geq 0$, \begin{enumerate*}[label=(\alph*)]
		\item $M^{-1}\dist{x}{y}-A  \leq \dist{f(x)}{f(y)} \leq M\dist{x}{y}+A$ for every $x,y\in V(G)$; and \item for every $z\in V(H)$ there is $x\in V(G)$ such that $\dist{z}{f(x)}\leq A$. 
	\end{enumerate*}
	The graphs $G$ and $H$ are \emph{quasi-isometric} if there exists a quasi-isometric map $f$ between $V(G)$ and $V(H)$. 
\end{definition}

We highlight the definition of quasi-isometry with additive distortion. 

\begin{definition}\label{def:add}
	A class $\mathcal{G}$ of graphs admits a \emph{quasi-isometry with additive distortion} to a class $\mathcal{H}$ if there exists a constant $A$  (depending only on $\mathcal{G}$ and $\mathcal{H}$) such that any $G\in \mathcal{G}$ admits a $(1,A)$-quasi-isometry to a graph in $\mathcal{H}$.
\end{definition}

A \emph{graph-partition of}~$G$ \emph{over}~$H$, or \emph{$H$-partition} for short, is a partition $\mathcal{H} := (V_h : h \in V(H))$ of $V(G)$ indexed by the nodes of $H$ such that for every edge $uv \in E(G)$, if $u \in V_{g}$ and $v \in V_h$, then $g = h$ or $gh \in E(H)$. We say $\mathcal{H}$ is \emph{honest}, if $V_h$ is non-empty for all $h \in V(H)$ and if for every edge $gh \in E(H)$ there exists an edge $uv \in E(G)$ such that $u \in V_g$ and $v \in V_h$. 
We say that $\mathcal{H}$ is \emph{$R$-bounded}, if each $V_h$ has diameter at most~$R$.

\begin{lemma}[\cite{albrechtsen2025excluding}] \label{lem part QI}
	Let $H, G$ be graphs, and let $\mathcal{H}$ be an honest, $R$-bounded $H$-partition  of $G$ for some $R\in \mathbb{R}$. Then $G$ is $(R + 1, R/(R+1))$-quasi-isometric to~$H$. 
\end{lemma}

A \emph{rooted} graph is a pair $(H,s)$ where $H$ is a graph and $s$ is one of its vertices, called its \emph{root}. We will sometimes omit $s$ from the notation if it is clear from the context. A rooted graph $(H,s)$ has a natural layering: we denote by $L^i = L^i_{H,s} := \{h \in V(H) : d_H(s,h) = i\}$ the \emph{$i$-th layer} of $H$. Given a vertex $h \in V(H)$ we denote by $\heightn h = \heightn {h,s}$ the unique integer satisfying $h \in L^{\heightn h}$. 

Let $\mathcal{H}= (H, (V_h)_{h \in V(H)})$ be a graph-partition of a graph~$G$ over a graph~$H$ which is also rooted at a node $s$. For every $n \in \mathbb{N}$, let $G^n=G^n_{\mathcal{H}}$  denote the  subgraph of~$G$ induced by those vertices that are contained in partition classes~$V_h$ of nodes $h$ in the layers of $H$ up to $L^n$, i.e.\ $G^n := G[\bigcup_{i \leq n} \bigcup_{h \in L^i} V_h]$. For a node $h\neq s$ of $H$, we let \emph{$\delta^\downarrow_h$} (resp. \emph{$\upboundary{h}$}) be the set of vertices of $V_h$ that are adjacent to some vertex of $G^{\heightn h-1}$ (resp. $G^{\heightn h+1} - G^{\heightn h}$). The \emph{height $R_h$} of a node $h$ of $H$ is the maximum distance 
$\max_{v\in V_h} d_G\left(\delta^\downarrow_h,v\right)$ of vertices from its `bottom' $\delta^\downarrow_h$. We say that $V_h$ is \emph{level}, if  $V_h = B_{G-G^{ \heightn h-1}}\left(\downboundary{h}, R_h\right)$.

\begin{definition}\label{def:nice}
	For an integer $R$, let $\mathcal{H}= (H, (V_h)_{h \in V(H)})$ be an $R$-bounded honest graph partition of a graph~$G$ over a graph~$H$ which is also rooted at a node $s$. We say $\mathcal{H}$ is ``\emph{nicely rooted}'' if 
	\begin{enumerate}[label=(\alph*),noitemsep,topsep=0pt]
		\item\label{it:independent} for all $j\in \mathbb{N}$, the layer $L^j$ is an independent set,
		
		\item\label{it:level} there exists a vertex $r\in V_s$ such that $V_s=B_G(r,R_s)$ with $R_s\le R$, and for every node $h\neq s$, $V_h$ is level,
		
		\item\label{it:same-height} the heights of two nodes $ x,y\in V(H)$ lying in the same cluster w.r.t $s$ are the same. 
	\end{enumerate}
\end{definition}

\begin{lemma}\label{lem:down-bd-dst}
	For an integer $R$, let $\mathcal{H}= (H, (V_h)_{h \in V(H)})$ be an $R$-bounded nicely rooted honest graph-partition of a graph~$G$ over a graph~$H$ which is also rooted at a node $s$. Let $r\in V_s$ be a vertex such that $V_s=B_G(r,R_s)$. Let $g$ and $h$ be two nodes of $H$ that lie in the same cluster of $H$ w.r.t $s$. For any $w\in \downboundary{h}, x\in \downboundary{g}$, $\distance{G}{r}{w} = \distance{G}{r}{x}$. Moreover, for any $y\in \upboundary{h}, z\in \upboundary{g}$, $\distance{G}{r}{y} = \distance{G}{r}{z}$.
\end{lemma}
\begin{proof}
	We prove the lemma by induction on the layer that contains $g$ and $h$. Assume $g,h\in L^1$. For two vertices $w\in \downboundary{g}, x\in \downboundary{h}$, there exist $w',x'\in V_s$ such that $\{ww',xx'\}\subseteq E(G)$. 
	By Definition~\ref{def:nice}\ref{it:level}, $\distance{G}{r}{w'} = \distance{G}{r}{x'}=R_s$.  Hence, $\distance{G}{r}{w} = \distance{G}{r}{w'}+1 = R_s+1$ and $\distance{G}{r}{x} = \distance{G}{r}{x'}+1 = R_s+1$. Hence, $\distance{G}{r}{w} = \distance{G}{r}{x}$. 
	Suppose $y\in \upboundary{g},z\in \upboundary{h}$. As any $(r,y)$-isometric path contains a vertex of $\downboundary{g}$ and by the above arguments, all vertices of $\downboundary{g}$ are equidistant from $r$ in $G$, it follows that $\distance{G}{r}{y} = \distance{G}{r}{\downboundary{g}} + R_g$. Similarly,  $\distance{G}{r}{z} = \distance{G}{r}{\downboundary{h}} + R_h$. 
	Since $g$ and $h$ lie in the same cluster w.r.t $s$, $R_h=R_g$ (due to Definition~\ref{def:nice}\ref{it:same-height}). By the above arguments $\distance{G}{r}{\downboundary{g}} = \distance{G}{r}{\downboundary{h}} = R_s+1$. Hence,   $\distance{G}{r}{y} = \distance{G}{r}{z} = R_s+1+R_h$.   
	
	For an integer $k\geq 1$, assume the statement is true for all pairs of nodes that lie in the same cluster w.r.t $s$ and lie in $L^i$ with $i\leq k$. Now consider two nodes $g$ and $h$ that lie in the same cluster w.r.t $s$ and $g,h\in L^{k+1}$. Let $w\in \downboundary{g},x\in \downboundary{h}$ and let $g',h'\in L^k$ be nodes such that there exist $w'\in V_{g'}, x'\in V_{h'}$ with $ww'\in E(G)$ and $xx'\in E(G)$. Observe that, $g'$ and $h'$ lie in the same cluster w.r.t $s$ and $w'\in \upboundary{g'}, x'\in \upboundary{h'}$. Hence, $\distance{G}{r}{w'}=\distance{G}{r}{x'}$. Since $\distance{G}{r}{w} = \distance{G}{r}{w'}+1$ and $\distance{G}{r}{x} = \distance{G}{r}{x'}+1$, we have that $\distance{G}{r}{w}=\distance{G}{r}{x}$. 
	
	For the second part, consider two vertices $y\in \upboundary{g}, z\in \upboundary{h}$. As any $(r,y)$-isometric path contains a vertex of $\downboundary{g}$ and by the above arguments, all vertices of $\downboundary{g}$ are equidistant from $r$ in $G$, it follows that $\distance{G}{r}{y} = \distance{G}{r}{\downboundary{g}} + R_g$. Similarly,  $\distance{G}{r}{z} = \distance{G}{r}{\downboundary{h}} + R_h$. 
	Since $g$ and $h$ lie in the same cluster w.r.t $s$, $R_h=R_g$ (due to Definition~\ref{def:nice}\ref{it:same-height}) and $\distance{G}{r}{\downboundary{g}} = \distance{G}{r}{\downboundary{h}}$.
\end{proof}

\begin{lemma}\label{lem:dst-incerase}
	For an integer $R$, let $\mathcal{H}= (H, (V_h)_{h \in V(H)})$ be an $R$-bounded nicely rooted honest graph-partition of a graph~$G$ over a graph~$H$ which is also rooted at a node $s$. Let $r\in V_s$ be a vertex such that $V_s=B_G(r,R_s)$. For $gh\in E(H)$ with $\heightn g=\heightn h-1$ and $u \in V_g, v\in V_h$, $\distance{G}{r}{u} < \distance{G}{r}{v}$.
\end{lemma}
\begin{proof}
 Let $P$ be an $(r,v)$-isometric path in $G$. 
	Let $X=\displaystyle\bigcup\limits_{g'\in L^{\heightn g}} \upboundary{g'}$ and by definition (of nicely rooted graph partitions), $X$ is an $(r,v)$-separator in $G$. 
	Hence, there exists a node $g'$ in $L^{\heightn g}$ such that $P$ contains a vertex $w_1\in \upboundary{g'}$. Observe that $\distance{G}{r}{v} > \distance{G}{r}{w_1}$. Moreover,  $g,g'$ lie in the same cluster w.r.t $s$ in $H$. By Lemma~\ref{lem:down-bd-dst}, for a vertex $w_2 \in \upboundary{g}$, $\distance{G}{r}{w_1} = \distance{G}{r}{w_2}$. Now observe that $\distance{G}{r}{u} \leq \distance{G}{r}{w_2} = \distance{G}{r}{w_1} < \distance{G}{r}{v}$. 
\end{proof}

For an isometric path in $G$, we define its \emph{trace} in $H$ as follows. 

\begin{definition}
	For an integer $R$, let $\mathcal{H}= (H, (V_h)_{h \in V(H)})$ be an $R$-bounded honest graph partition of a graph $G$ over a graph $H$. For $u,v\in V(G)$ and a $(u,v)$-isometric path $P$ in $G$, $\trc{H}{P}$ is a sequence  $\left(g_1,g_2,\ldots,g_k\right)$ of distinct elements such that for $i\in [k], g_i\in V(H)$, $u\in V_{g_1}, v\in V_{g_k}$, and for $i\in [k-1]$, there is an edge $e_i=u_iv_i\in E(G)$ such that $u_i\in V_{g_i}\cap V(P), v_i\in V_{g_{i+1}}\cap V(P)$.
\end{definition}

In the next lemma we show that, when a graph $G$ has a nicely rooted honest graph partition over a graph $H$, the trace of an ``$r$-rooted isometric path'' in $G$ is an ``$s$-rooted isometric path'' in $H$.

 \begin{lemma}\label{lem:trace}
 	For an integer $R$, let $\mathcal{H}= (H, (V_h)_{h \in V(H)})$ be an $R$-bounded nicely rooted honest graph-partition of a graph~$G$ over a graph~$H$ which is also rooted at a node $s$. Let $r\in V_s$ be a vertex such that $V_s=B_G(r,R_s)$. 
 	\begin{enumerate}[label=(\alph*)]
 		\item\label{it:trace} For any $u\in V(G), h\in V(H), u\in V_h$ and an $(r,u)$-isometric path $P$ in $G$, $\trc{H}{P}$ is an $(s,h)$-isometric path in $H$.
 		
 		\item\label{it:rooted-pair-preserve} Let $(x,y)$ be an $r$-rooted pair in $G$ and $g,h\in V(H)$ such that $x\in V_g, y\in V_h$. Then $(g,h)$ is an $s$-rooted pair in $H$.
 	\end{enumerate}
 	\end{lemma}
 	
 \begin{proof}
 	First we prove \ref{it:trace}. Let $P$ be an $(r,u)$-isometric path in $G$. We prove that
 		for any node $a\in V(H)$, $V_a\cap V(P)$ induces a subpath of $P$. 
 	
 		To prove the above claim, suppose there exists a node $a\in V(H)$ such that the graph $F$ induced by $V_a\cap V(P)$ has at least two components. For distinct connected components $C,C'$ of $F$, let $P(C,C')$ be the minimal subpath of $P$ that has vertices adjacent to both $C$ and $C'$. 
 		Let $C_1$ and $C_2$ be two components of $F$ such that none of the vertices of $P(C_1,C_2)$ lie in $V_a$. 
 		Without loss of generality, assume the vertices of $C_1$ are closer to $r$ in $G$ than those of $C_2$. 
 		In other words, for $u\in V(C_1), v\in V(C_2)$, $\distance{G}{r}{u}<\distance{G}{r}{v}$. Let $Q=P(C_1,C_2)$. For $i\in \{1,2\}$, let 
 		$f_i=u_iv_i\in E(G)$ such that $u_i\in V(C_i)$ and $v_i\in V(Q)$.  Observe that $\distance{G}{r}{u_1} < \distance{G}{r}{w}$ for all $w\in V(Q)$. 
 		
 		If $v_1$ lies in a node $b\in V(H)$ where $\heightn {b}=\heightn a-1$, then by Lemma~\ref{lem:dst-incerase}, $\distance{G}{r}{v_1} < \distance{G}{r}{u_1}$, which is a contradiction. Hence, $v_1$ lies in a node $b\in V(H)$ where $\heightn {b}=\heightn a+1$. Observe that $u_1\in \upboundary{a}$.
 		Let $w\in V(Q)$ be the vertex which is closest to $u_1$ in $G$ and is adjacent to some vertex $w_1$ where $w_1\in V_{c}\setminus \{u_1\}, c\in L^{\heightn a}$. 
 		Since $Q$ is a connected  subgraph which has vertices adjacent to both $C_1$ and $C_2$,  $w$ always exists. Observe that $w_1\in \upboundary{c}$. 
 		Moreover, $a$ and $c$ lie in the same cluster w.r.t $s$ in $H$. Hence, by Lemma~\ref{lem:dst-incerase} and Lemma~\ref{lem:down-bd-dst}, $\distance{G}{r}{u_1}=\distance{G}{r}{w_1}$, which is a contradiction. 
 
 	\medskip\noindent Let $\trc{H}{P}=(a_1,a_2,\ldots,a_k)$. Since for $i\in [k]$, $V_{a_i}\cap V(P)$ induces a subpath of $P$,  observe that, for $j\in [2,k]$, $a_j$ is the unique node of $H$ that contains the endpoint of the path induced by $V(P)\setminus \{V_{a_1}\cup \ldots \cup V_{a_{j-1}}\}$ which is also adjacent to a vertex of $V_{a_{j-1}}$. Next,  we prove that 
 		for $j\in [k-1]$,  $\distance{H}{s}{a_{j+1}} = \distance{H}{s}{a_j}+1$. 
 
 	 To prove the above claim, suppose for some $j\in [k-1]$, $\distance{H}{s}{a_{j+1}}\leq \distance{H}{s}{a_j}$. Let $w_1w_2\in E(P)$ be the unique edge such that $w_1\in V_{a_j}\cap V(P)$ and $w_2\in V_{a_{j+1}}\cap V(P)$. Clearly, $\distance{H}{s}{a_{j+1}}=\distance{H}{s}{a_{j}}-1$ and $\distance{G}{r}{w_1} < \distance{G}{r}{w_2} $. But this contradicts Lemma~\ref{lem:dst-incerase}. 
 	Hence, $\trc{H}{P}$ is an $(s,h)$-isometric path.
 	
 	\medskip \noindent To prove \ref{it:rooted-pair-preserve}, let $P$ be an $(r,y)$-isometric path in $G$ that contains $x$. By \ref{it:trace}, $Q=\trc{H}{P}$ is an $(s,h)$-isometric path in $H$. Since $x\in V(P)$, $g\in V(Q)$. Hence $(g,h)$ is an $s$-rooted pair in $H$.
 \end{proof}

\subsection{Strong isometric path complexity and additive quasi-isometry}

Given a graph $G$ and a vertex $r$ of $G$, a set $S$ of isometric paths of $G$ is \emph{$r$-rooted} if $r$ is one of the end-vertices of all the isometric paths in $S$. A path $P$ is \emph{covered} by a set $\mathcal{Q}$ of paths if $V(P)\subseteq \displaystyle\bigcup\limits_{Q\in \mathcal{Q}} V(Q)$. For a graph $G$ and a vertex $r$, $\ipco{r,G}$ is the minimum integer $k$ such that  any isometric path $P$ of $G$ can be covered by a set of $k$ many $r$-rooted paths. 

\begin{definition}[\cite{chakraborty2026isometric,  chakraborty2025strong}]\label{def:strong}
	The ``\emph{strong isometric path complexity}'' of a connected graph $G$, denoted by $\sipco{G}$, is the minimum integer $k$ such that for every vertex $r\in V(G)$, $\ipco{r,G}\leq k$. 
\end{definition}

\begin{theorem}[\cite{chakraborty2025strong}]\label{th:sipco-K_2,t}
	The strong isometric path complexity of graphs that do not contain $K_{2,t}$ as a $K$-fat minor is at most $(48K+5)t-1$.
\end{theorem}

For a graph class $\mathcal{G}$, $\sipco{\mathcal{G}} = \max\{ \sipco{G}:G\in \mathcal{G} \}$.  The next lemma follows from the proof of a lemma of \cite{chakraborty2025k}, but we provide a proof for completeness.

	\begin{lemma}\label{lem:sipco-distort}
	Let $X$ and $Y$ be two graphs, $f\colon V(X)\rightarrow V(Y)$ be a function, and $r\in V(X)$ be a vertex. Let $t\geq 1, k\geq 1$ be two integers such that the following hold. \begin{enumerate}[label=(\roman*)]
		\item For any edge $ab\in E(X)$, $\distance{Y}{f(a)}{f(b)}\leq k+1$;
		\item \sloppy For any $r$-rooted pair $(a,b)$ in $X$,  $\distance{Y}{f(a)}{f(b)}\leq \distance{X}{a}{b} + t$.
	\end{enumerate} Then, for $\{u,v\}\subseteq V(X)$, $\distance{Y}{f(u)}{f(v)}\leq \distance{X}{u}{v} + t\cdot \sipco{X} + k(\sipco{X}-1)$.
\end{lemma}
\begin{proof}
	Let $P$ be any $(u,v)$-isometric path in $X$ and $\mathcal{Q}$ be a minimum cardinality set of $r$-rooted isometric paths in $X$ that covers $P$ and for any $Q\in \mathcal{Q}$, $|V(P)\cap V(Q)|$ is maximized. Let $z\coloneq |\mathcal{Q}|$ and by definition, $z\leq \sipco{X}$. Observe that for any $Q\in \mathcal{Q}$, the path induced by  $V(P)\cap V(Q)$ is an isometric path in $X$.  Moreover, there exists exactly one path $Q_1\in \mathcal{Q}$ that contains $u$. Otherwise, if there exists a path $Q'\in \mathcal{Q}\setminus \{Q_1\}$ containing $u$, then either $V(P)\cap V(Q')$ contains $V(P)\cap V(Q_1)$, or vice versa.  But this contradicts the minimality of $\mathcal{Q}$.
	
	Let $Q_1\in \mathcal{Q}$ be the path containing $u$, and define $P_1\coloneq P[V(P)\cap V(Q_1)]$. 
	Let $a_1\coloneq u$, and $b_1$ be the end-vertex of $P_1$ distinct from $a_1$, unless $P_1$ is reduced to a single vertex in which case $b_1=a_1$.  
	
	For $i\in\{2,\ldots, z\}$, let $Q_i\in \mathcal{Q}$ be a path that contains the end-vertex of the path induced by $V(P)\setminus \displaystyle\bigcup\limits_{j<i} V(Q_j)$ which is distinct from $v$. For $i\in\{2,\ldots, z\}$, let $P_i= P[V(P)\cap V(Q_i)]$, and $a_i$ be the end-vertex of $P_i$ which is adjacent to $b_{i-1}$ in $P$, and let $b_i$ be the end-vertex of $P_i$ distinct from $a_i$. Note that $b_z=v$.
	Hence,
	
	\begin{equation*} \label{eq1}
		\begin{split}
			\distance{Y}{f(u)}{f(v)} & = \distance{Y}{f(a_1)}{f(b_z)} \\
			& \leq \displaystyle\sum\limits_{i=1}^{z} \distance{Y}{f(a_i)}{f(b_i)} + \displaystyle\sum\limits_{i=1}^{z-1} \distance{Y}{f(b_i)}{f(a_{i+1})}\\
			& \leq \displaystyle\sum\limits_{i=1}^{z} (\distance{X}{a_i}{b_i}+t) + \displaystyle\sum\limits_{i=1}^{z-1} (k+1) \hspace{10pt} \\
			& \leq   \displaystyle\sum\limits_{i=1}^{z} \distance{X}{a_i}{b_i} + \displaystyle\sum\limits_{i=1}^{z-1} \distance{X}{b_i}{a_{i+1}}+ t\cdot z + k(z-1) \\
			& \leq  \distance{X}{a_1}{b_z} + t\cdot z + k(z-1) \\
			& =   \distance{X}{u}{v} + t\cdot z + k(z-1)
		\end{split}
	\end{equation*}
	This completes the proof.
\end{proof}

\begin{lemma}\label{lem:attach-ipco}
Let $H$ be a graph such that $\sipco{\langle H\rangle}\leq k$. Then, for any path $P$ of $H$ and a component $C$ of $H-V(P)$, $|N_H(C)\cap V(P)|\leq k$.
\end{lemma}
\begin{proof}
Assume for contradiction that there is a path $P$ in $H$ and a component $C$ of $H-V(P)$ such that $|N_H(C)\cap V(P)|\geq k+1$.
For each $u\in N_H(C)\cap V(P)$, let $r_u$ be a vertex of $C$ adjacent to $u$. 
Let $T$ be a spanning tree of $C$ and $r$ be any vertex of $C$. Set $t_1=|E(P)|$, $t_2=\max_{u\in V(C)} \distance{T}{r}{u}$, $t_3=t_1+t_2+2$ and $t_4=2t_3+t_1+2$. 
Define a function $\phi\colon E(H)\rightarrow \mathbb{Z}^+$ as follows.
\begin{enumerate}[label=(\roman*),topsep=0pt]
	\item For each $e\in E(T)\cup E(P)$, define $\phi(e)=1$.
	\item  For each $u\in N_H(C)\cap V(P)$, define $\phi(ur_u)=t_3-\distance{T}{r}{r_u}$. 
	\item For any edge $e$ that does not satisfy the above two conditions, define $\phi(e)=t_4$.  
\end{enumerate}

Construct a graph $H^*$ by replacing each edge $e$ of $H$ with a path of length $\phi(e)$. Clearly, $H^*\in \langle H \rangle$. 
We prove the following claim. 

\begin{claim}\label{clm:W}
		The path $P$ is isometric in $H^*$. 
\end{claim} 
\begin{subproof}
	By definition, $P$ is a path in $H^*$.  Let $x,y$ be the endpoints of $P$ and $Q$ be any $(x,y)$-path distinct from $P$. By definition of $H^*$, $Q$ must contain a subpath of length at least $t_1+2$ and hence $|E(Q)|>|E(P)|$.   
\end{subproof}

\begin{claim}\label{clm:ipcok+1}
	Any $r$-rooted isometric path in $H^*$ contains at most one vertex of $N_H(C)\cap V(P)$.
\end{claim}
\begin{subproof}
	By construction of $H^*$, for any $u\in N_H(C)\cap V(P)$, $\distance{H^*}{r}{u}= t_3$. (Indeed,  for any $u\in N_H(C)\cap V(P)$, there is an $(r,u)$-path in $H^*$ through $r_u$ that has length $t_3$ and any other $(r,u)$-path will have length at least $t_4>t_3$.) Hence, any path of $H^*$ that has $r$ as an endpoint and contains two vertices of $N_H(C)\cap V(P)$ will have length at least $t_3+1$, and thus not isometric in $H^*$.
\end{subproof}

By Claims~\ref{clm:W} and \ref{clm:ipcok+1}, $\ipco{r,H^*}\geq k+1$. This implies $\sipco{H^*}\geq k+1$, a contradiction.
\end{proof}

\section{Proof of Theorem~\ref{thm:additive}}

For the remainder of this section, let $G$ be a graph. For an integer $R\geq 1$, let $\mathcal{H}= (H, (V_h)_{h \in V(H)})$ be a nicely rooted $R$-bounded graph-partition of $G$ over a graph $H$. Let $s$ be the root node of $H$. Furthermore, for an integer $k\geq 1$, let the strong isometric path complexities of both $G$ and $\langle H \rangle$ be at most $k$. 
Next, we construct a graph $F$.

\subsection{Construction of $F$.}\label{sec:construct}
By Definition~\ref{def:nice}, there is a vertex $r$ 
such that $V_s=B_G(r,R_s)$. Define $r_s=r$ and call $r$ to be the \emph{root} of $F$.  
For each node $h\in V(H)$ with $h\neq s$, 
let $r_h$ be any vertex of $\delta^{\downarrow}_h$.  For each $h\in V(H)$, $r_h$ is the \emph{representative} of $h$ and the partition class $V_h$.
Let $p(s)=s$ and for each  node $h\in V(H)$ with $h\neq s$, let $p(h)$ denote a node $g\in V(H)$ such that $gh\in E(H)$ and $\heightn g=\heightn h-1$. 
For each edge $e=gh\in E(H)$ with $\heightn g=\heightn h-1$, introduce a path $P_e$ of length $d_G(r,r_h) - d_G(r,r_g)$.  
For each $h\in V(H)$ and vertex $v\in V_h\setminus \{r_h\}$, introduce a path $Q_v$ of length $d_G(r,v) - \distance{G}{r}{r_{p(h)}}$. 
Define 

$$\mathcal{S} = \{P_e\colon e\in E(H)\} \cup \{Q_v : h \in V(H), v \in V_h \setminus {r_h}\}$$  Due to
Lemma~\ref{lem:dst-incerase}, all paths in $\mathcal{S}$ have positive lengths, and therefore are well-defined. 
Now define 

$$X=V(G)\displaystyle\bigcup\limits_{P\in \mathcal{S}} V(P)$$

We complete the construction of $F$ by identifying some vertices of $X$ as follows. 
For each $e=gh\in E(H)$ identify one endpoint of $P_e$ with $r_g$ and the other with $r_h$. 
For each $h\in V(H)$ and vertex $v\in V_h\setminus \{r_h\}$, identify one endpoint of $Q_v$ with $v$ and the other with $r_{p(h)}$. 
This completes the construction of $F$.  
The meaning of all notations introduced above will remain consistent for the entirety of this section.
Next we establish some properties of $F$.

\subsection{Structure of $F$}

First we prove the following lemma.

\begin{lemma}\label{lem:sipco-F}
	Treewidth of $F$ is at most that of $H$ and $\sipco{F}\leq \sipco{\langle H\rangle}+2$.  
\end{lemma}
\begin{proof}
	
	For every node $g$ and vertex $v\in V_g\setminus \{r_g\}$, delete from $F$, all vertices of $Q_v\setminus \{r_{p(g)}\}$. The resulting graph, say $F^*$, is a subdivision of $H$. Since subdivision of edges preserves treewidth, we have that $tw(F)\leq tw(H)$.
	 For the second part, let $k=\sipco{\langle H\rangle }$ and $P$ be any isometric path of $F$. Let $E_1$ be the set of edges of  $P$ that lie in some pendant path of $F$. Clearly $E_1$ induces at most two components $C_1,C_2$ and $E(P)\setminus E_1$ induces an isometric path $Q$ of $F^*$.	Let $v$ be any vertex of $F$. 
	
	Suppose $v\in V(F^*)$. Since $F^*\in \langle H \rangle$, $\sipco{F^*}\leq k$ i.e. vertices of $Q$ can be covered with at most $k$ many $v$-rooted isometric paths in $F^*$, and thus in $F$. Since $C_1,C_2$ are subpaths of pendant paths, each of them can be covered with one $v$-rooted isometric path in $F$. Therefore, $P$ can be covered with $k+2$ many $v$-rooted isometric paths. 
	
	Suppose $v\not\in V(F^*)$. If $v$ is incident to some edge in $C_1$ or $C_2$, then $P$ can be covered with two $v$-rooted paths. Otherwise,  there exists a vertex $v'\in F^*$ such that any isometric path in $F$ containing $v$ as an endpoint also contains $v'$. Since $F^*\in \langle H \rangle$, $\sipco{F^*}\leq k$,  similar arguments as above imply that $P$ can be covered with $k+2$ many $v$-rooted paths.
\end{proof}


\subsection{Upper bounds on distances of rooted pairs of $G$}
In the next lemma, we establish upper bounds on the lengths of the paths in $\mathcal{S}$ and the distances (in $F$) between the representatives of nodes of $H$.

\begin{lemma}\label{dist:rep}
	The graph $F$ satisfies the following. 
	\begin{enumerate}[label=(\alph*),topsep=0pt]
		\item\label{it:dr1} For an edge $e=gh\in E(H)   $ with $\heightn g=\heightn h-1$, $\distance{G}{r_g}{r_h}\leq 2R+1$, the length of $P_e$ is at most $2R+1$, and
		 $\distance{F}{r_g}{r_h} \leq 2R+1$.
		
		\item\label{it:dr2} For a node $h\in V(H)$, and a vertex $v\in V_h\setminus \{r_h\}$, the length of $Q_v$ is at most $3R+1$ and, $\distance{F}{v}{r_{p(h)}}\leq3R+1$.
		
		\item\label{it:dr4} For any $v\in V(G), h\in V(H)$ with $v\in V_{h}$, $\distance{F}{v}{r_h} \leq 5R+2$.
		
		\item\label{it:dr5} For an edge $uv\in E(G)$, $\distance{F}{u}{v} \leq 12R+5$.
	\end{enumerate}
\end{lemma}

\begin{proof}
To prove \ref{it:dr1}, assume without loss of generality $\heightn g=\heightn h-1$. By definition, $r_h$ is a vertex of $\downboundary{h}$ and there exists $x\in V_g, x'\in \downboundary{h}$ such that $xx'\in E(G)$. Since $\distance{G}{x}{r_g}\leq R$ and $\distance{G}{x'}{r_h}\leq R$, we have $\distance{G}{r}{r_h} - \distance{G}{r}{r_g} \leq \distance{G}{r_h}{r_g}\leq 2R+1$.  Now the statement follows from the definition of $P_e$ and the definition of $F$.

\medskip \noindent To prove \ref{it:dr2}, observe that  $\distance{G}{v}{r_{p(h)}} \leq \distance{G}{v}{r_{h}} +  \distance{G}{r_{h}}{r_{p(h)}}$. By definition, $\distance{G}{v}{r_{h}} \leq R $. Now \ref{it:dr1} implies $ \distance{G}{r}{v} - \distance{G}{r}{r_{p(h)}} \leq \distance{G}{v}{r_{p(h)}} \leq 3R+1$. Now the statement \ref{it:dr2} follows from the definition of $Q_v$, and the definition of $F$.

	\medskip \noindent	To prove \ref{it:dr4}, observe that $\distance{F}{v}{r_h} \leq \distance{F}{v}{r_{p(h)}} + \distance{F}{r_h}{r_{p(h)}}$. Now \ref{it:dr1} and \ref{it:dr2} imply $\distance{F}{v}{r_h} \leq 5R+2$. 
	
	\medskip \noindent	To prove \ref{it:dr5}, consider an edge $xy\in E(G)$. First assume there exists a vertex $h\in V(H)$ such that $x,y\in V_h$.
	Observe that $\distance{F}{x}{y} \leq \distance{F}{x}{r_h} + \distance{F}{r_h}{y}$. By \ref{it:dr4}, $\distance{F}{x}{y}\leq 10R+4$. 
	Now consider the case when there is an edge $e=gh\in E(H)$ with $\heightn g=\heightn h-1$ and $x\in V_h, y\in V_g$. Observe that $\distance{F}{x}{y}\leq \distance{F}{x}{r_h} + \distance{F}{r_h}{r_g} + \distance{F}{r_g}{y}$. Now \ref{it:dr1} and \ref{it:dr4} imply, $\distance{F}{x}{y}\leq 12R+5$.
\end{proof}

\begin{lemma}\label{lem:no-increase}
	For any $s$-rooted pair $(g,h)$ in $H$, $\distance{F}{r_g}{r_h} \leq \distance{G}{r_g}{r_h}$. 
\end{lemma}
\begin{proof}
	If $g=h$, the lemma is trivially true. So, assume $g\neq h$ and let $(g,h)$  be an $s$-rooted pair in $H$ and $P$ be any $(s,h)$-isometric path in $H$ that contains $g$. Let $e_1,e_2,\ldots,e_k$ be the edges of the $(g,h)$-subpath of $P$ where $e_1$ is incident to $g$, for $i\in [k-1]$, $e_i,e_{i+1}$ share an endpoint and $e_k$ is incident to $h$. For $i\in [k]$, let $e_i=g_ih_i$ such that $\distance{H}{s}{g_i}=\distance{H}{s}{h_i}-1$. Observe that, $g=g_1, h=h_k$, and $i\in [k-1]$, $h_i=g_{i+1}$. Thus,
	\begin{align*}
		\distance{F}{r_g}{r_h} & \leq \displaystyle\sum\limits_{i=1}^k \distance{F}{r_{g_i}}{r_{h_i}}\\
		& \leq \displaystyle\sum\limits_{i=1}^k \distance{G}{r}{r_{h_i}} - \distance{G}{r}{r_{g_i}} & [\text{Definition of }F]\\
		& = \distance{G}{r}{r_{h_k}} - \distance{G}{r}{r_{g_1}}\\
		& \leq \distance{G}{r_{h}}{r_g}
	\end{align*}
	This completes the proof.
\end{proof}

Next, we show that the distances between two vertices forming an $r$-rooted pair in $G$ do not increase by too much in $F$.

\begin{lemma}\label{lem:G->F}
	There is a function $\beta\colon V(G)\rightarrow V(F)$, such that all of the following hold. 
	\begin{enumerate}[label=(\alph*),topsep=0pt]
		\item\label{it:distG-1} For each $x\in V(F)$, there is a vertex $y\in V(G)$ such that $\distance{F}{x}{\beta(y)} \leq 3R+1$. 
		
		\item\label{it:distG-2} For each edge $xy\in E(G)$, $\distance{F}{\beta(x)}{\beta(y)} \leq 12R+5$. 
		
		\item\label{it:distG-3} For any $r$-rooted pair $(x,y)$ in $G$, $\distance{F}{\beta(x)}{\beta(y)} \leq \distance{G}{x}{y} + 12R+4$.
	\end{enumerate}
\end{lemma}
\begin{proof}
	By definition, any vertex of $G$ is also a vertex of $F$. For $x\in V(G)$, define $\beta(x)=x$. Hence for $x,y\in V(G)$,  $\distance{F}{\beta(x)}{\beta(y)}=\distance{F}{x}{y}$.

	\medskip \noindent To prove \ref{it:distG-1}, any vertex of $F$ is contained in some path in $\mathcal{S}$, 
	and that the endpoints of the paths in $F$ are identified with vertices of $V(G)$. Now  \ref{it:distG-1} follows from Lemma~\ref{dist:rep}\ref{it:dr1} and \ref{it:dr2}. 
	
	\medskip \noindent Statement \ref{it:distG-2} follows from Lemma~\ref{dist:rep}\ref{it:dr5}.
	
	\medskip\noindent To prove \ref{it:distG-3}, consider two vertices $x,y\in V(G)$ such that there is an $(r,y)$-isometric path in $G$ that contains $x$. Let $g,h\in V(H)$ be the vertices of $H$ such that $x\in V_g$ and $y\in V_h.$
	If $g=h,$ then $\distance{F}{x}{y} \leq \distance{F}{x}{r_g} + \distance{F}{r_g}{y}$. 
	Lemma~\ref{dist:rep}\ref{it:dr4} implies $\distance{F}{x}{y} \leq 10R+4$. Otherwise, $g\neq h$, and by Lemma~\ref{lem:trace}\ref{it:rooted-pair-preserve}, there is a $(s,h)$-isometric path in $H$ that contains $g$ i.e. $(g,h)$ is an $s$-rooted pair in $H$.  Therefore, \begin{align*}
		\distance{F}{x}{y} &  \leq	\distance{F}{x}{r_g} + \distance{F}{r_g}{r_h} + \distance{F}{r_h}{y} \\
		& \leq 	10R+4 + \distance{F}{r_g}{r_h} & [\text{Lemma~\ref{dist:rep}\ref{it:dr4}}] \\ 
		& \leq 10R+4 + \distance{G}{r_g}{r_h} & [\text{Lemma~\ref{lem:no-increase}}]\\
		& \leq  10R+4 + \distance{G}{r_g}{x} + \distance{G}{x}{y} + \distance{G}{y}{r_h} \\ 
		& \leq 12R+4 + \distance{G}{x}{y} & [V_g,V_h~\text{have diameter at most }R ]
	\end{align*}
	This completes the proof. 
\end{proof}

\subsection{Upper bounds on distances of rooted pairs of $F$}
%

We identify certain pairs of vertices whose distances are the same in $G$ and $F$. 

\begin{lemma}\label{lem:equal-GF-rooted}
	Let	$(g,h)$ be an $s$-rooted pair in $H$ such that $(r_g,r_h)$ is an $r$-rooted pair in $G$. Then, $\distance{G}{r_g}{r_h} = \distance{F}{r_g}{r_h}$. 
\end{lemma}
\begin{proof}
	The case $g=h$ is trivial, and hence assume $g\neq h$. Let $P$ be an $(r_g,r_h)$-isometric path in $F$, and let
	$r_{g_1},r_{g_2},\ldots,r_{g_k}$ be the ordering of the representative vertices encountered on $P$ when traversed from $g_1=g$ to $g_k=h$.
	By the construction of $F$,
	\[
	|E(P_{e_i})|
	=d_G(r,r_{g_{i+1}})-d_G(r,r_{g_i}).
	\]
	Therefore,
	\begin{align*}
		d_F(r_g,r_h)=|E(P)|	& =\sum_{i=1}^{k-1}|E(P_{e_i})|\\
		&=\sum_{i=1}^{k-1}
		d_G(r,r_{g_{i+1}})-d_G(r,r_{g_i}) \\
		&=d_G(r,r_h)-d_G(r,r_g)\\
		&=d_G(r_g,r_h),
	\end{align*}
	where the last equality holds because $(r_g,r_h)$ is an
	$r$-rooted pair in $G$.
	Together with Lemma \ref{lem:no-increase}, we have the proof.
\end{proof}

\begin{lemma}\label{lem:rooted-rep-F-dist}
	Let $g,h\in V(H)$ be two nodes such that $(r_g,r_h)$ is an $r$-rooted pair in $F$. Then, $\distance{G}{r_g}{r_h}\leq \distance{F}{r_g}{r_h} + 2R\cdot \sipco{\langle H \rangle}$.
\end{lemma}
\begin{proof}
Let $k=\sipco{\langle H\rangle }$.	If $g=h$, then the lemma is trivially true. If $g=s$, then the lemma is true by Lemma~\ref{lem:equal-GF-rooted}. Therefore, we assume $s\neq g\neq h$.
	Let $P$ be a $(r,r_h)$-isometric path in $F$ that contains $r_g$.
	Let $Q=\left(g_1,g_2,\ldots,g_t\right)$ be the sequence such that for $i\in [t], g_i\in V(H)$, $s=g_1, h= g_t$, and for $i\in [t-1]$, $P$ contains a $(r_{g_i},r_{g_{i+1}})$-subpath that does not contain any vertex of $G$ apart from $r_{g_i}$ and $r_{g_{i+1}}$.
	 Clearly, $Q$ is an $(s,h)$-isometric path in $H$ that contains $g$. Therefore, $(g,h)$ is an $s$-rooted pair in $H$. Let $Q_1$ be the $(g,h)$-subpath of $Q$.
	Applying Lemma~\ref{lem:attach-ipco} on $Q_1$ and the component $C$ of $H-V(Q_1)$ that contains $s$, we have that $|N_H(C)\cap V(Q_1)|\leq k$. Let $A=N_H(C)\cap V(Q_1)$.
	
	\begin{claim}\label{clm:hop}
		For each node $h'\in V(Q_1)\setminus \{g\}$, there exists a node $g'\in A\setminus \{h'\}$ such that $(g',h')$ is an $s$-rooted pair in $H$ and $\distance{G}{r_{g'}}{r_{h'}} \leq \distance{G}{r}{r_{h'}} - \distance{G}{r}{r_{g'}} + 2R$. 
	\end{claim} 
	\begin{subproof}
		Fix a node $h'\in V(Q_1)\setminus \{g\}$ and let $h_1$ be the node of $Q_1$ where $\distance{Q_1}{g}{h_1}=\distance{Q_1}{g}{h'}-1$. Since $H$ is an honest partition of $G$, there exists an edge $uv\in E(G)$ such that $v\in V_{h'}$ and $u\in V_{h_1}$. Observe that $v\in \downboundary{h'}$ and by  Lemma~\ref{lem:down-bd-dst},
		\begin{align}\label{eq1}
			\distance{G}{r}{v}=\distance{G}{r}{r_{h'}}
		\end{align} 
		By Lemma~\ref{lem:dst-incerase}, $\distance{G}{r}{v}=\distance{G}{r}{u}+1$, and therefore there exists an $(r,v)$-isometric path $T$ in $G$ that contains $u$. 
		Let $Q_2=\trc{H}{T}$ and $h_2\in V(Q_2)\cap V(Q_1)$ be the node closest to $s$. Observe that $h_2\neq h'$, and $h_2\in A$. 
		Let $w$ be the vertex of $V(T)\cap V_{h_2}$ that is closest to $r$ in $G$. Observe that $w\in \downboundary{h_2}$ and by  Lemma~\ref{lem:down-bd-dst},
		\begin{align}\label{eq2}
			\distance{G}{r}{w}=\distance{G}{r}{r_{h_2}}
		\end{align} 
		
		
		\noindent By (\ref{eq1}), (\ref{eq2}) we have that, 
		\begin{align*}
			\distance{G}{r_{h_2}}{r_{h'}} & \leq  \distance{G}{r_{h_2}}{w} + \distance{G}{w}{v} + \distance{G}{v}{r_{h'}} & [\text{Triangle inequality}] \\
			& \leq 2R+ \distance{G}{w}{v} & [\text{Definitions of }H,\mathcal{H}] \\
			& = 2R + \distance{G}{r}{v} - \distance{G}{r}{w} & [\text{Definition of }w \text{ and }v] \\
			& = 2R + \distance{G}{r}{r_{h'}} - \distance{G}{r}{r_{h_2}}
		\end{align*}
		Setting $g'=h_2$ proves the claim.
	\end{subproof}
	
	Let $(g'_1, g'_2, \ldots g'_{t'})$ be the maximum subsequence of $Q_1$ such that $g'_{t'}=h$, and for each $i\in [t'-1,1]$, $g'_{i}$ is the node closest to $s$ that satisfies Claim~\ref{clm:hop} when applied to $g'_{i+1}$. Observe that, $t'-1\leq k$. Observe that, $g'_1=g$. Hence, 
	\begin{align*}
		\distance{G}{r_g}{r_h} & = \distance{G}{r_{g'_1}}{r_{g'_{t'}}}
		\leq \displaystyle\sum\limits_{i=1}^{t'-1} \distance{G}{r_{g'_i}}{r_{g'_{i+1}}} & [\text{Triangle inequality}] \\
		& \leq \displaystyle\sum\limits_{i=1}^{t'-1} \left(\distance{G}{r}{r_{g'_{i+1}}} - \distance{G}{r}{r_{g'_i}}\right) + 2R(t'-1) & [\text{Claim~\ref{clm:hop}}] \\
		& = \distance{G}{r}{r_{g'_{t'}}} - \distance{G}{r}{r_{g'_1}} + 2R(t'-1) \\
		& = \distance{F}{r_{g'_1}}{r_{g'_{t'}}} + 2R(t'-1) & [\text{Construction of }F] \\
		& \le \distance{F}{r_{g}}{r_h} + 2Rk
	\end{align*}
	This completes the proof.
\end{proof}

\begin{lemma}\label{lem:F->G}
There is a function $\alpha\colon V(F)\rightarrow V(G)$ satisfying the following.
	\begin{enumerate}[label=(\alph*),topsep=0pt]
	\item  For each edge $xy\in E(F)$, $\distance{G}{\alpha(x)}{\alpha(y)}\leq 3R+1$.
	\item For any $r$-rooted pair $(x,y)$ in $F$, $\distance{G}{\alpha(x)}{\alpha(y)} \leq \distance{F}{x}{y} + 4R+2+2R\cdot \sipco{\langle H\rangle}$. 
\end{enumerate}
\end{lemma}
\begin{proof}
The function $\alpha\colon V(F)\rightarrow V(G)$ is defined as follows. 
Let $x\in V(F)$ be a vertex. If $x$ is also a vertex of $G$, then define $\alpha(x)=x$. 
Otherwise, if there exist $y\in V(G), h\in V(H)$ such that $y\in V_h\setminus \{r_h\}, x\in V(Q_y)$ where $Q_y\in \mathcal{S}$, then define $\alpha(x)=r_{p(h)}$. 
Otherwise there exists $e=gh\in E(H)$ with $\heightn g=\heightn h-1$ and $x\in V(P_e)$ where $P_e\in \mathcal{S}$. In this case define $\alpha(x)=r_g$. This concludes the definition of $\alpha$.

Using $\alpha$, we define a function $\alpha'\colon  \mathcal{S}\rightarrow 2^{V(G)}$ as follows. Let $Q\in \mathcal{S}$ be a path. 
Assume for some $e'=gh\in E(H)$ with $\heightn g=\heightn {h}-1$, $Q=P_{e'}$. Define $\alpha'(Q)=\{r_g,r_h\}$.
Now assume for some $z\in V(G)$ and $h'\in V(H)$, $Q=Q_z$ and $z\in V_{h'}$.  In this case, define $\alpha'(Q)=\{z,r_{p(h')}\}$.
By Lemma~\ref{dist:rep}\ref{it:dr1} and \ref{it:dr2}, for any $Q\in \mathcal{S}$ and $\{a,b\}=\alpha'(Q)$, $\distance{G}{a}{b}\leq 3R+1$.

\medskip \noindent To prove $(a)$, let $e=xy$ be any edge of $F$.  Observe that there exists a path $T_e\in \mathcal{S}$ such that $e\in E(T_e)$.  By definition, $\{\alpha(x),\alpha(y)\}\subseteq \alpha'(T_e)$. 
The above arguments imply, $\distance{G}{\alpha(x)}{\alpha(y)}\leq 3R+1$.  

\medskip \noindent To prove $(b)$, let $x,y\in V(F)$ such that there is an $(r,y)$-isometric path $Q$ in $F$ that contains $x$.  If $x=y$, then the lemma is trivially true.
For \(x\ne y\), let \(e_x\) be the edge immediately after \(x\) toward \(y\) along \(Q\), and \(e_y\) be the edge immediately before \(y\). Let $T_x,T_y\in \mathcal{S}$ be two paths containing $e_x$ and $e_y$, respectively.
 If $T_x=T_y$, then $\{\alpha(x),\alpha(y)\}\subseteq \alpha'(T_x)$ and thus $\distance{G}{\alpha(x)}{\alpha(y)}\leq 3R+1$.  
 Assume $T_x\neq T_y$. Then there exists $e_1=g_1h_1\in E(H)$ with $\heightn {g_1}=\heightn {h_1}-1$ such that $P_{e_1}=T_x$. By definition, $\alpha(x)\in \{r_{g_1}, r_{h_1}\}$. 

\medskip \noindent \textbf{Case 1.} Suppose there exists $e_2=g_2h_2\in E(H)$ with $\heightn {g_2}=\heightn {h_2}-1$ such that $P_{e_2}=T_y$. Then $\alpha(y)\in \{r_{g_2},r_{h_2}\}$. 
Since $T_x\neq T_y$, $Q$ contains $r_{g_1}, x, r_{h_1}, r_{g_2},y$. 
This implies $Q$ contains $\alpha(x),x,\alpha(y),y$ and $\distance{F}{r}{\alpha(x)} \leq \distance{F}{r}{x} \leq \distance{F}{r}{\alpha(y)} \leq \distance{F}{r}{y}$. 
Thus, 
\begin{align*}
	\distance{G}{\alpha(x)}{\alpha(y)} & \leq \distance{F}{\alpha(x)}{\alpha(y)} + 2R\cdot\sipco{\langle H\rangle} & [\text{Lemma~\ref{lem:rooted-rep-F-dist}}] \\
	& =  \distance{F}{\alpha(x)}{x} +  \distance{F}{x}{y} - \distance{F}{\alpha(y)}{y} + 2R\cdot\sipco{\langle H\rangle} \\ 
	& \leq 2R\cdot\sipco{\langle H\rangle}  + 2R + 1 + \distance{F}{x}{y}.
\end{align*}
 
 \medskip \noindent \textbf{Case 2.} Suppose there exists $h_2\in V(H)$, and $z\in V_{h_2}\setminus \{r_{h_2}\}$ such that $Q_{z}=T_y$. Then $\alpha(y)\in \{y,r_{p(h_2)}\}$. 
 Since $T_x\neq T_y$, and $Q$  contains $r_{g_1}, x, r_{h_1}, r_{p(h_2)},y$ and $\distance{F}{r}{r_{g_1}}\leq \distance{F}{r}{x} \leq \distance{F}{r}{r_{h_1}}\leq \distance{F}{r}{r_{p(h_2)}} \leq \distance{F}{r}{y}$. 
 This implies $Q$ contains $\alpha(x),x,\alpha(y),y$ and $\distance{F}{r}{\alpha(x)} \leq \distance{F}{r}{x} \leq \distance{F}{r}{\alpha(y)} \leq \distance{F}{r}{y}$.  
If $\alpha(y)=r_{p(h_2)}$ then, 
 \begin{align*}
 	\distance{G}{\alpha(x)}{\alpha(y)} & \leq  \distance{F}{\alpha(x)}{\alpha(y)} + 2R\cdot\sipco{\langle H\rangle}  & [\text{Lemma~\ref{lem:rooted-rep-F-dist}}] \\
 	& =  \distance{F}{\alpha(x)}{x} +  \distance{F}{x}{y} - \distance{F}{\alpha(y)}{y}+2R\cdot\sipco{\langle H\rangle}  \\ 
 	& \leq 2R + 1 + \distance{F}{x}{y}+2R\cdot\sipco{\langle H\rangle} .
 \end{align*}
 Otherwise, $\alpha(y)=y$ and thus,    \begin{align*}
 	\distance{G}{\alpha(x)}{\alpha(y)} &  \leq	\distance{G}{\alpha(x)}{r_{p(h_2)}} + \distance{G}{r_{p(h_2)}}{y} \\
 	& \leq 	\distance{G}{\alpha(x)}{r_{p(h_2)}} + 2R+1 \\ 
 	& \leq \distance{F}{\alpha(x)}{r_{p({h_2})}} + 2R+1 + 2R\cdot\sipco{\langle H\rangle}  \hspace{1cm} [\text{Lemma~\ref{lem:rooted-rep-F-dist}}]\\
 	& =  \distance{F}{\alpha(x)}{x} +  \distance{F}{x}{y} - \distance{F}{r_{p({h_2})}}{y} + 2R+1 + 2R\cdot\sipco{\langle H\rangle}  \\ 
 	& \leq 4R + 2 + 2R\cdot\sipco{\langle H\rangle} + \distance{F}{x}{y}.
 \end{align*}
This completes the proof. 
\end{proof}

\subsection{Completion of the proof.}

Recall that $G,H$ are two graphs such that strong isometric path complexities of both $G,\langle H \rangle$ are at most $k$, $G$ admits an honest, nicely rooted $R$-bounded $H$-partition. Consider the graph $F$ constructed  in Section~\ref{sec:construct}. By Lemma~\ref{lem:sipco-F}, $tw(F)\leq tw(H)$.

By Lemmas~\ref{lem:sipco-distort} and \ref{lem:G->F} , we have for all $u,v\in V(G)$, $\distance{F}{u}{v} \leq (12R+4)k+(12R+5)(k-1) + \distance{G}{u}{v}$. This implies for all $u,v\in V(G)$, $\distance{F}{u}{v} \leq 33\cdot R\cdot k + \distance{G}{u}{v}$.

By Lemma~\ref{lem:sipco-F}, $\sipco{F}\leq k+2$.
By Lemmas~\ref{lem:sipco-distort} and \ref{lem:F->G} , we have for all $u,v\in V(G)$, $\distance{G}{u}{v} \leq (4R+2+2Rk)(k+2)+(3R+1)(k+1) + \distance{F}{u}{v}$. This implies for all $u,v\in V(G)$, $\distance{G}{u}{v} \leq 33\cdot R\cdot k^2 + \distance{F}{u}{v}$. This completes the proof of  Theorem~\ref{thm:additive}.

We also obtain the following corollary. For two graphs $X,Y$, we say $X$ has $Y$ as a core if it is possible to obtain a subdivision of $Y$ from $X$ by repeatedly deleting vertices of degree $1$. 

\begin{corollary}\label{cor:subdivision}
		For $k\geq 1, R\geq 1$, let $G,H$ be two graphs such that $\sipco{G}\leq k, \sipco{\langle H \rangle}\leq k$, and $G$ admits an honest, nicely rooted $R$-bounded $H$-partition.  There is a graph $F$ that has $H$ as a core and $G$ admits a $(1,33\cdot R \cdot k^2)$-quasi-isometry to $F$. 
\end{corollary}

\section{Proof of Theorem~\ref{thm:K2t}}

For $t\geq 3$, let $\Theta_t$ denote the multigraph consisting of two vertices and $t$ edges between them.
A result of Albrechtsen, Distel, and Georgakopoulos~\cite{albrechtsen2025excluding} implies the following. 

\begin{lemma}\label{lem:albr}
	There is a function $f:\mathbb{N}\times \mathbb{N}\rightarrow \mathbb{N}$ such that for any graph $G$ without any $K$-fat $K_{2,t}$-minor, there is an honest, nicely rooted $f(K,t)$-bounded $H$-partition of $G$ where $H$ is $\Theta_t$ minor-free. 
\end{lemma}
\begin{proof}
    Apply Lemma~3.4 of \cite{albrechtsen2025excluding} with $3K.$ Since $G$ contains no $K$-fat $K_{2,t}$-minor, Lemma~2.1 of \cite{albrechtsen2025excluding} implies that $G$ contains no $3K$-fat $\Theta_t$-minor. Hence, Lemma~3.4 of \cite{albrechtsen2025excluding} provides an honest, $R(t,3K)$-bounded $H$-partition satisfying properties (i)-(iv) of \cite{albrechtsen2025excluding}. 

    We root $H$ at the node $s$ used in the construction of \cite{albrechtsen2025excluding}. Property~\ref{it:independent} of our Definition~\ref{def:nice} is exactly (i) of \cite{albrechtsen2025excluding}. Property~\ref{it:level} follows from (ii) of \cite{albrechtsen2025excluding} : the root class is a ball around the distinguished root vertex, while every other class is level.

    It remains to check \ref{it:same-height}, consider two nodes $h_A,h_B\in L^{n+1}$ that belong to the same cluster with respect to $s.$ By the construction provided by Lemma~3.4 of \cite{albrechtsen2025excluding}, the partition $\mathcal{P}$ is the union of the partitions $\mathcal{P}_D,$ where $\mathcal{D}$ ranges over the classes of the partition $\mathcal R$ of the components of $G-G^n.$ Moreover, for every component $C\in \mathcal{P}_D,$ the height $R_C$ of $C$ depends on the component $D$ such that $C\in \mathcal{P}_D.$ Thus, it is enough to prove $A$ and $B$ belong to $\mathcal{P}_D$ for the same $D\in \mathcal{R}.$ 

    Suppose, for a contradiction, that $D_A\neq D_B.$ Recall that $h_A,h_B\in L^{n+1}.$ If at least one of $D_A,D_B$ is a singleton $\{C\},$ then, by the definition of $\mathcal{R},$ $C$ has neighbors in exactly one partition class $V_x$ for some vertex $x\in L^n.$ In this case, one of $V_{h_A}, V_{h_B}$ is contained in $C$, and $V_x$ separates $V_{h_A}$ from $V_{h_B}$ in $G.$ Consequently, every path between $h_A$ and $h_B$ in $H$ contains $x,$ and hence $h_A$ and $h_B$ lie in distinct components of $H-B_H(s,n)$.   
    
    Otherwise, both $D_A$ and $D_B$ are of the form $D_Z$ for some components $Z$ of $G-G_{n-1}.$ Since $D_A\neq D_B,$ the corresponding components $Z_A$ and $Z_B$ are distinct. Moreover, $V_{h_A}\subseteq Z_A$ and $V_{h_B}\subseteq Z_B.$ Hence any path in $H$ between $h_A$ and $h_B$ must contain a node of $L^n,$ otherwise, by the definition of an $H$-partition, the corresponding bags would yield a path in $G-G_{n-1}$ joining $Z_A$ and $Z_B,$ a contradiction. Thus again $h_A$ and $h_B$ lie in distinct components of $H-B_H(s,n).$

    In both cases, $h_A$ and $h_B$ cannot belong to the same cluster with respect to $s$, since they both lie in $L^{n+1}.$ This contradiction shows that $D_A=D_B.$ Therefore 
    $$R_{h_A}=R_A=R_B=R_{h_B},$$
    Hence, \ref{it:same-height} follows.

    Finally, if $H$ contained $\Theta_t$ as a minor, then, since $\Theta_t$ is $2$-connected, Lemma~3.3 of \cite{albrechtsen2025excluding} would yield a $3K$-fat $\Theta_t$-minor in $G,$ a contradiction. Thus $H$ is $\Theta_t$-minor-free. Taking $f(K,t)=R(t,3K)$ proves the lemma.
\end{proof}

Let $G$ be a graph that does not contain a $K$-fat $K_{2,t}$-minor for some integers $K\geq 1, t\geq 3$. Now, Corollary~\ref{cor:subdivision} provides a graph $F$ that has $H$ as a core and $G$ admits a $(1,33\cdot R\cdot t_1^2)$-quasi-isometry to $F$. Here we assume the notations used to construct $F$ in Section~\ref{sec:construct}.
Suppose $F$ has a $K_{2,t}$-minor model $\mathcal{M}$ where $X$ and $Y$ represent the branch sets corresponding to the high degree vertices of the $K_{2,t}$-minor. Observe that, there exists at least one node $x\in V(H)$ such that $r_x\in X$. Similarly, there exists at least one node $y\in V(H)$ such that $r_y\in Y$. Let $X'$ be the set of nodes of $H$ whose representatives lie in $X$ i.e. $X'=\{x\in V(H)\colon r_x\in X\}$. Similarly, define $Y'=\{y\in V(H)\colon r_y\in Y\}$. Let $Z_1,Z_2,\ldots,Z_t$ be the sets of $\mathcal{M}$ corresponding to the degree two vertices of $K_{2,t}$. Since $Z_1,Z_2,\ldots,Z_t$ induce disjoint connected subgraphs in $F$ each seeing $X$ and $Y$, it is possible to construct a $\Theta_t$-minor where $X'$ and $Y'$ correspond to the vertices of $\Theta_t$. This contradicts Lemma~\ref{lem:albr}. This completes the proof of Theorem~\ref{thm:K2t}.

By Theorem~\ref{th:sipco-K_2,t}, there exists a constant $t_1$, depending only on $K$ and $t$, such that $G$ has strong isometric path complexity at most $t_1$. By Lemma~\ref{lem:albr}, there exists a constant depending on $K$ and $t$, say $R$, such that $G$ admits an honest, nicely rooted $R$-bounded $H$-partition of $G$ where $H$ is $\Theta_t$-minor-free. From above arguments it follows that all graphs in $\langle H \rangle$ are $K_{2,t}$-minor-free. By Theorem~\ref{th:sipco-K_2,t} and the fact that $H$ has no $K_{2,t}$-minor, there exists a constant $t_2$ depending only on $K$ and $t$, such that $t_2\leq t_1$ and strong isometric path complexity of $\langle H \rangle$ is at most $t_2$.

\section{Conclusion}

In this paper, we establish some conditions under which an $R$-bounded graph partition of a graph $G$ over a graph $H$ can be used to obtain quasi-isometry with additive distortion.
As a corollary, we obtain that $K_{2,t}$-asymptotic minor-free graphs admit  quasi-isometries with additive distortion to the class of $K_{2,t}$-minor-free graphs. Let $U_t$ denote the graph obtained by adding a universal vertex to a path on $t$ vertices. Since $U_t$-asymptotic minor-free graphs have bounded strong isometric path complexity, following is a natural question.

\begin{question}
	Is it true that if a graph $X$ has no $K$-fat $U_t$-minor for some $K \in \mathbb{N}, t\geq 3$, then $X$ admits a map of bounded additive distortion onto a graph with no $U_t$-minor?
\end{question}

While we believe the answer to the above question is yes, we do not know the existence of a nicely rooted graph partition with bounded diameter for $U_t$-asymptotic minor-free graphs. 
 Another natural question is whether a similar statement can be made for $K_4$-asymptotic minor-free graphs, also raised by Nguyen, Seymour, and Scott~\cite{nguyen2025asymptotic}. However, $K_4$-asymptotic minor-free graphs do not have bounded strong isometric path complexity, which makes the approach of Theorem~\ref{thm:additive} ineffective.

\bibliographystyle{plain}
\bibliography{references}

\begin{thebibliography}{10}

\bibitem{albrechtsen2025counterexample}
S.~Albrechtsen and J.~Davies.
\newblock Counterexample to the conjectured coarse grid theorem.
\newblock {\em arXiv preprint arXiv:2508.15342}, 2025.

\bibitem{albrechtsen2023structural}
S.~Albrechtsen, R.~Diestel, A.~Elm, E.~Fluck, Raphael~W. Jacobs, P.~Knappe, and
  P.~Wollan.
\newblock A structural duality for path-decompositions into parts of small
  radius.
\newblock {\em arXiv preprint arXiv:2307.08497}, 2023.

\bibitem{albrechtsen2025excluding}
S.~Albrechtsen, M.~Distel, and A.~Georgakopoulos.
\newblock Excluding $ {K}_{2, t}$ as a fat minor.
\newblock {\em arXiv preprint arXiv:2510.14644}, 2025.

\bibitem{albrechtsen2025characterisation}
S.~Albrechtsen, R.~W. Jacobs, P.~Knappe, and P.~Wollan.
\newblock A characterisation of graphs quasi-isometric to {$K_4$}-minor-free
  graphs.
\newblock {\em Combinatorica}, 45(6):61, 2025.

\bibitem{albrechtsen2026small}
Sandra Albrechtsen, Marc Distel, and Agelos Georgakopoulos.
\newblock Small counterexamples to the fat minor conjecture.
\newblock {\em arXiv preprint arXiv:2601.05761}, 2026.

\bibitem{berger2024bounded}
E.~Berger and P.~Seymour.
\newblock Bounded-diameter tree-decompositions.
\newblock {\em Combinatorica}, pages 1--16, 2024.

\bibitem{bonamy2023asymptotic}
M.~Bonamy, N.~Bousquet, L.~Esperet, C.~Groenland, C.~H. Liu, F.~Pirot, and
  A.~Scott.
\newblock Asymptotic dimension of minor-closed families and {A}ssouad-{N}agata
  dimension of surfaces.
\newblock {\em Journal of the European Mathematical Society},
  26(10):3739--3791, 2023.

\bibitem{brandstadt1999distance}
A.~Brandst{\"a}dt, V.~Chepoi, and F.~Dragan.
\newblock Distance approximating trees for chordal and dually chordal graphs.
\newblock {\em Journal of Algorithms}, 30(1):166--184, 1999.

\bibitem{chakraborty2025k}
D.~Chakraborty.
\newblock ${K}_{2, 3}$-induced minor-free graphs admit quasi-isometry with
  additive distortion to graphs of tree-width at most two.
\newblock {\em arXiv preprint arXiv:2503.00798}, 2025.

\bibitem{chakraborty2026isometric}
D.~Chakraborty, J.~Chalopin, F.~Foucaud, and Y.~Vax{\`e}s.
\newblock Isometric path complexity of graphs.
\newblock {\em Discrete Mathematics}, 349(2):114743, 2026.

\bibitem{c22isom}
D.~Chakraborty, A.~Dailly, S.~Das, F.~Foucaud, H.~Gahlawat, and S.~K. Ghosh.
\newblock Complexity and algorithms for {ISOMETRIC} {PATH} {COVER} on chordal
  graphs and beyond.
\newblock In {\em {ISAAC}}, volume 248 of {\em LIPIcs}, pages 12:1--12:17,
  2022.

\bibitem{chakraborty2025strong}
D.~Chakraborty and F.~Foucaud.
\newblock Strong isometric path complexity of graphs: Asymptotic minors,
  restricted holes, and graph operations.
\newblock {\em arXiv:2501.10828}, 2025.

\bibitem{chepoi2000note}
V.~Chepoi and F.~Dragan.
\newblock A note on distance approximating trees in graphs.
\newblock {\em European Journal of Combinatorics}, 21(6):761--766, 2000.

\bibitem{chepoi2012constant}
V.~Chepoi, F.~F. Dragan, I.~Newman, Y.~Rabinovich, and Y.~Vax{\`e}s.
\newblock Constant approximation algorithms for embedding graph metrics into
  trees and outerplanar graphs.
\newblock {\em Discrete \& Computational Geometry}, 47:187--214, 2012.

\bibitem{davies2026fat}
J.~Davies, R.~Hickingbotham, F.~Illingworth, and R.~McCarty.
\newblock Fat minors cannot be thinned (by quasi-isometries).
\newblock {\em Analysis and Geometry in Metric Spaces}, 14(1):20250036, 2026.

\bibitem{dourisboure2007tree}
Y.~Dourisboure and C.~Gavoille.
\newblock Tree-decompositions with bags of small diameter.
\newblock {\em Discrete Mathematics}, 307(16):2008--2029, 2007.

\bibitem{dragan2025graph}
F.~F. Dragan.
\newblock Graph parameters that are coarsely equivalent to tree-length.
\newblock {\em arXiv:2502.00951}, 2025.

\bibitem{georgakopoulos2025graph}
A.~Georgakopoulos and P.~Papasoglu.
\newblock Graph minors and metric spaces.
\newblock {\em Combinatorica}, 45(3):33, 2025.

\bibitem{nguyen2025asymptotic}
T.~Nguyen, A.~Scott, and P.~Seymour.
\newblock Asymptotic structure. {II}. {P}ath-width and additive quasi-isometry.
\newblock {\em arXiv preprint arXiv:2509.09031}, 2025.

\bibitem{papasoglu2026additivequasiisometriescacti}
P.~Papasoglu and E.~Swenson.
\newblock Additive quasi-isometries and cacti.
\newblock {\em Arxiv: 2609.15917}, 2026.

\end{thebibliography}

\end{document}